\documentclass[12pt]{amsart}

\newtheorem{theorem}{Theorem}[section]
\newtheorem{Theorem}[theorem]{Theorem}
\newtheorem{Problem}[theorem]{Problem}
\newcommand{\BProblem}{\begin{Problem}}
\newcommand{\EProblem}{\end{Problem}}
\newtheorem{Vermutung}[theorem]{Vermutung}
\newtheorem{Folgerung}[theorem]{Folgerung}
\newtheorem{Conjecture}[theorem]{Conjecture}
\newtheorem{Proposition}[theorem]{Proposition}
\newtheorem{Corollary}[theorem]{Corollary}
\newtheorem{Korollar}[theorem]{Korollar}
\newtheorem{Lemma}[theorem]{Lemma}
\newtheorem{Satz}[theorem]{Satz}
\newcommand{\BSatz}{\begin{Satz}}
\newcommand{\ESatz}{\end{Satz}}

\newtheorem{Hauptsatz}[theorem]{Hauptsatz}
\newcommand{\BHauptsatz}{\begin{Hauptsatz}}
\newcommand{\EHauptsatz}{\end{Hauptsatz}}

\newtheorem{Klassifikationssatz}[theorem]{Klassifikationssatz}
\newcommand{\BKlassifikationssatz}{\begin{Klassifikationssatz}}
\newcommand{\EKlassifikationssatz}{\end{Klassifikationssatz}}
\newtheorem{Hilfssatz}[theorem]{Hilfssatz}
\newcommand{\BHilfssatz}{\begin{Hilfssatz}}
\newcommand{\EHilfssatz}{\end{Hilfssatz}}

\newtheorem{Corollary (of the proof)}[theorem]{Corollary (of the proof)}

\theoremstyle{definition}
\newtheorem{Definition}[theorem]{Definition}

\theoremstyle{remark}
\newtheorem{Remark}[theorem]{Remark}

\newcommand{\ERemark}{\end{Remark}}
\newcommand{\BRemark}{\begin{Remark}}
\newtheorem{Remarks}[theorem]{Remarks}
\newcommand{\ERemarks}{\end{Remarks}}
\newcommand{\BRemarks}{\begin{Remarks}}

\newtheorem{Ramifications}[theorem]{Ramifications}
\newcommand{\BRamifications}{\begin{Ramifications}}
\newcommand{\ERamifications}{\end{Ramifications}}
\newtheorem{Bemerkung}[theorem]{Bemerkung}
\newcommand{\BBemerkung}{\begin{Bemerkung}}
\newcommand{\EBemerkung}{\end{Bemerkung}}

\newtheorem{Anschauung}[theorem]{Anschauung}
\newcommand{\BAnschauung}{\begin{Anschauung}}
\newcommand{\EAnschauung}{\end{Anschauung}}

\newtheorem{Bemerkungen}[theorem]{Bemerkungen}
\newcommand{\BBemerkungen}{\begin{Bemerkungen}}
\newcommand{\EBemerkungen}{\end{Bemerkungen}}
\newtheorem{Kommentar}[theorem]{Kommentar}
\newcommand{\BKommentar}{\begin{Kommentar}}
\newcommand{\EKommentar}{\end{Kommentar}}
\newtheorem{Notation}[theorem]{Notation}
\newcommand{\BNotation}{\begin{Notation}}
\newcommand{\ENotation}{\end{Notation}}
\newtheorem{Behauptung}[theorem]{Behauptung}
\newcommand{\BBehauptung}{\begin{Behauptung}}
\newcommand{\EBehauptung}{\end{Behauptung}}
\newtheorem{Beispiel}[theorem]{Beispiel}
\newcommand{\BBeispiel}{\begin{Beispiel}}
\newcommand{\EBeispiel}{\end{Beispiel}}

\newtheorem{Motivation}[theorem]{Motivation}
\newcommand{\BMotivation}{\begin{Motivation}}
\newcommand{\EMotivation}{\end{Motivation}}

\newtheorem{Beispiele}[theorem]{Beispiele}
\newcommand{\BBeispiele}{\begin{Beispiele}}
\newcommand{\EBeispiele}{\end{Beispiele}}
\newtheorem{Example}[theorem]{Example}
\newcommand{\EExample}{\end{Example}}
\newcommand{\BExample}{\begin{Example}}
\newtheorem{Examples}[theorem]{Examples}
\newcommand{\EExamples}{\end{Examples}}
\newcommand{\BExamples}{\begin{Examples}}
\newtheorem{Exercise}[theorem]{Exercise}
\newcommand{\EExercise}{\end{Exercise}}
\newcommand{\BExercise}{\begin{Exercise}}
\newtheorem{Ubung}[theorem]{"Ubung}
\newcommand{\EUbung}{\end{Ubung}}
\newcommand{\BUbung}{\begin{Ubung}}
\newtheorem{Ubungen}[theorem]{"Ubungen}
\newcommand{\EUbungen}{\end{Ubungen}}
\newcommand{\BUbungen}{\begin{Ubungen}}
\newtheorem{Exercises}[theorem]{Exercises}
\newcommand{\EExercises}{\end{Exercises}}
\newcommand{\BExercises}{\begin{Exercises}}
\newtheorem{Claim}[theorem]{Claim}

\newcommand{\BTheorem}{\begin{Theorem}}
\newcommand{\ETheorem}{\end{Theorem}}
\newcommand{\BVermutung}{\begin{Vermutung}}
\newcommand{\EVermutung}{\end{Vermutung}}
\newcommand{\BFolgerung}{\begin{Folgerung}}
\newcommand{\EFolgerung}{\end{Folgerung}}
\newcommand{\BConjecture}{\begin{Conjecture}}
\newcommand{\EConjecture}{\end{Conjecture}}
\newcommand{\BProposition}{\begin{Proposition}}
\newcommand{\EProposition}{\end{Proposition}}
\newcommand{\BCorollary}{\begin{Corollary}}
\newcommand{\ECorollary}{\end{Corollary}}
\newcommand{\BKorollar}{\begin{Korollar}}
\newcommand{\EKorollar}{\end{Korollar}}
\newcommand{\BDefinition}{\begin{Definition}}
\newcommand{\EDefinition}{\end{Definition}}
\newcommand{\BLemma}{\begin{Lemma}}
\newcommand{\ELemma}{\end{Lemma}}
\newcommand{\BClaim}{\begin{Claim}}
\newcommand{\EClaim}{\end{Claim}}

\newcommand{\da}{\downarrow}

\newcommand{\io}{\iota}

\newcommand{\la}{\lambda}

\newcommand{\al}{\alpha}

\newcommand{\ve}{\varepsilon}
\newcommand{\va}{\varphi}

\newcommand{\DZ}{{\mathbb Z}}

\numberwithin{equation}{subsection}
\begin{document}
\title{Littelmann's refined Demazure character formula revisited }
\author{Steen Ryom-Hansen}
\address{Matematisk Afdeling, Universitetsparken 5 \\
DK-2100 K{\o}benhavn \O, Danmark \\ steen@math.ku.dk }
\begin{abstract}{We give a purely combinatorial derivation 
of Littelmann's refined Demazure character formula. }
\end{abstract}

\maketitle

\thispagestyle{myheadings}
\font\rms=cmr8
\font\its=cmti8
\font\bfs=cmbx8
\markright{\its S\'eminaire Lotharingien de
Combinatoire \bfs 49 \rms (2003), Article~B49d\hfill}
\def\thepage{}

\section{Introduction}
The Demazure character formula is 
a generalization of Weyl's character 
formula. It was first stated by Demazure 
in [D], who showed that it would follow from a certain string 
property. However, it turned out that this property did not
hold in the original setting. The first correct 
proofs of the formula were therefore given by Andersen, [A] and 
Ramanan-Ramanathan [RR], using 
methods closely related to Frobenius splitting.

\medskip
This work is concerned with the crystal basis 
approach to the Demazure character formula.
In that setting the string property indeed {\it does} hold 
as demonstrated by 
Kashiwara in [K1]. We briefly review the deduction 
of the character formula from it. 

\medskip
We then go on to 
show that the string property can be obtained 
using only combinatorial properties of the crystals:
the Kashiwara operators $ \tilde{e}_i $, $ \tilde{f}_i $, together with 
the $*$-operation. 
This is different from the previous deductions of 
the formula, which use either a representation theoretical 
interpretation of the formula or appeal to Littelmann's path models. 
Our deduction should be contrasted with the remarks following 
Proposition {\bf 6.3.10} in Joseph's book, [J1].

\medskip

I would like to thank the referee for many useful suggestions.

\section{The refined Demazure character formula}

\subsection{}

Let us briefly recall the notion of crystal as introduced by 
Kashiwara. We refer to [K1,K2,J] for all unexplained notation. 
Let $ C:=(c_{i,j } )_{i,j \in I}  $ be a generalized Cartan matrix.
Crystals are certain combinatorial objects associated to $ C$.
They consist of a set $ B$ with maps 
$ \tilde{e}_i, \tilde{f}_i : B \rightarrow B \cup \{ 0 \}  $ 
and maps $ \epsilon_i, \varphi_i : B  \rightarrow { \mathbb Z } \cup
\{- \infty \} $, $  wt_i : B \rightarrow P $,   $ \forall
i \in I $, that satisfy certain conditions.

\medskip 
There is a crystal $B(\lambda) $ associated to 
the Weyl module $V(\lambda)$ of the quantized universal algebra 
$U_q({\mathfrak g})$. The limit crystal is called $B(\infty) $.

\medskip 

Given two crystals $B_1 $ and $ B_2 $ one can make $ B_1 \times B_2 $
into a crystal, which is called 
the tensor product $ B_1 \otimes B_2 $. For example we have that 

$$
 \tilde{f}_i(b_1 \otimes b_2)  = \left\{
\begin{array}{ll} 
\tilde{f}_i b_1 \otimes b_2
& \mbox{  if }   \varphi_i( b_1 ) >  \epsilon_i(b_2), \\ 
 b_1 \otimes \tilde{f}_ib_2
&    \mbox{  otherwise. } 
\end{array} \right.
$$

\pagenumbering{arabic}
\addtocounter{page}{1}
\markboth{\SMALL STEEN RYOM--HANSEN}{\SMALL LITTELMANN'S REFINED DEMAZURE 
CHARACTER FORMULA}

There is also a sum construction.
(But notice that not all crystals arise from 
the representation theoretical crystals using such constructions).

\medskip 

We shall mainly view crystals as combinatorial objects in the above sense,
but shall also appeal to Kashiwara's $*$-operation on $ B(\infty) $ (see [K1]).
We first of all need the following property:
for all $ i \in I $ there is an injective morphism of crystals $ \Psi_i : B(\infty ) \rightarrow 
B(\infty ) \otimes B_i $ where $ B_i $ is the crystal defined 
in example 1.2.6. of [K1]. It satisfies the following 
conditions
\begin{gather}~\label{eq1}
\Psi_i : u_{\infty} \mapsto u_{\infty} \otimes b_i ,
\\
\label{eq2}
\Psi_i ( \tilde{f}_i^{\ast} b ) = b^{\prime} \otimes \tilde{f}_i b^{\prime \prime} 
\text{ where  } \Psi_i ( b ) = b^{\prime} \otimes b^{\prime \prime},\\
\label{eq3}
\tilde{f}_i \Psi_i ( b) = \Psi_i ( \tilde{f}_i b ) \text{  and  } 
\tilde{e}_i \Psi_i ( b ) = \Psi_i( \tilde{e}_i b ) ,
\end{gather}
where $ u_{\infty} $ is the unique element of $ B_{\infty}  $ of
weight $ 0 $, and 
where $ B ( \infty ) \otimes B_i $ has the above structure of a 
tensor product. Joseph has given a purely combinatorial proof of the 
existence of $ \Psi_i$, [J2].

\medskip

Now, for a reduced expression 
$ s_{i_n}  s_{i_{n-1}} \ldots  s_{i_1} $ of the Weyl group element
$w$, 
we define $ B_w ( \infty ) \subset B ( \infty ) $ and 
$ B_w ( \la  ) \subset B ( \la ) $ in the following recursive way 
\begin{gather*}
   B_w ( \infty ) := \bigcup_k \tilde{f}^k_{i_{n}} B_{s_{i_n}w} (\infty),
\,\,\,\,\,\,\,B_1 ( \infty ) := \{u_{\infty } \}, \\
   B_w ( \la ) := \bigcup_k \tilde{f}^k_{i_{n}} B_{s_{i_n}w} (\la),
\,\,\,\,\,\,\, B_1 ( \la ) := \{ u_{\la } \} .
\end{gather*}
A priori, these definitions might depend on the choice of reduced expression
$ s_{i_n}  s_{i_{n-1}} \ldots  s_{i_1} $ of $w$. We shall later 
show that in fact $ B_w ( \infty ) $ and $ B_w ( \la ) $ are independent 
of this choice.

\subsection{}

Let $ {\mathcal D}_i $ be the additive operator on $ \DZ [ B( \la ) ] $ given by 
$$
{\mathcal D}_i b = \left\{
\begin{array}{lll} 
\sum_{ 0 \leq k \leq wt_i (b) } \tilde{f}^k_i b 
& \mbox{if} &   wt_i( b ) \geq 0, \\ 
 - \sum_{1 \leq k \leq -wt_i (b)-1 } \tilde{e}^k_i b 
& \mbox{if} &   wt_i( b ) < 0 .
\end{array} \right.
$$
Then the refined Demazure character formula, [K1,L1],
is the following equality 
in $ {\DZ} [B(\la)] $
\begin{equation}~\label{demazureformel}
 \sum_{b \in B_w ( \la ) } b = {\mathcal D}_{i_n} {\mathcal D}_{i_{n-1}} \ldots 
{\mathcal D}_{i_1} u_{\la} .
\end{equation}

The $ {\mathcal D }_i $'s induce the usual Demazure operators on the group 
ring of the weight lattice $ \DZ [P ] $ under the weight map 
$ w: \DZ [ B( \la ) ] \rightarrow \DZ [P ] $. Thus, if  $ W $ is 
finite and we take $ w = w_0 $ the 
longest element of the Weyl group, (\ref{demazureformel}) 
generalizes the original Demazure 
expression of the Weyl character, see e.g. [A]. 

\medskip

\subsection{}~\label{threeproperties}
In the rest of this section we shall review Kashiwara's proof of 
(\ref{demazureformel}). The idea is to 
reduce to the verification of the following three properties of 
$ B_w ( \la )$:

\begin{center}
\begin{enumerate}

\item $ B_w^{\ast}( \infty )  = B_{w^{-1}} ( \infty ) $,

\item $ {\tilde{e}}_i B_w ( \infty ) \subset B_w ( \infty ) \cup\{0 \} 
\,\,\,\,\,\, \forall i \in I $,

\item $ {\tilde{f}}_j b \in B_w  (\infty ) \Rightarrow \tilde{f}_j^k b \in 
B_w ( \infty ) \,\, \, \,\,  \,  \forall b \in B_w(\infty),  \forall k \in {\mathbb N},
\forall j \in I $ .

\end{enumerate} 
\end{center}

Let us denote any subset of  
$ B( \infty ) $ or of $ B(\la ) $ an  
$i$-string if it is of the form 
\begin{equation}
S = \{ \tilde{f}^k_i b \, | \, k \geq 0, \text{ where } b \in B(\la ) 
\text{ satisfies } \tilde{e}_i b = 0 \} .
\end{equation} 
We call $ b$ the highest weight vector of $ S $. The key ``Demazure
string property'' of these 
$i$-strings is then the following: for any $i$-string 
$ S \subset B( \infty ) $ we have that
\begin{equation}~\label{string}
B_w ( \infty ) \cap S \text{ is either } S \text{ or } \{ b \} 
\text { or the empty set } .
\end{equation} 
This is seen by combining (2) and (3). 

\subsection{}
The string property is also valid for $B(\lambda)$: to see this 
one defines for $ \la \in P $ 
the crystal on one element 
$ T_{\la } :=\{t_{\la}\} $ as follows:
\begin{gather*} wt_i ( t_{\la} ) = \langle \la , \al_i \rangle \,\,\,\,\,\,\,\,
\ve_i (  t_{\la} ) = - \infty ,  \,\,\,\,\,\,\,\, \varphi_i ( t_{\la} ) = 
-\infty ,\\
 \tilde{e}_i ( t_{\la} ) = 0 = \tilde{f}_i ( t_{\la} ) .
\end{gather*}
Let $ \la \in P^+ $. Then $ u_{\la} \mapsto u_{\infty } \otimes t_{\la} $
defines an embedding of crystals $ \io_{\lambda}: B(\la ) \mapsto B(\infty ) \otimes 
T_{\la} $ that commutes with the $ \tilde{e}_i $'s. 

\medskip
Now, $ B_w( \la ) $ is the inverse image of $ B_w ( \infty ) \otimes T_{\la} $
under $ \io_{\lambda}$. Furthermore, the inverse image under $ \io_{\lambda}$ of an 
$ i $-string for $ B(\infty ) $ is an $ i $-string for $ B(\la ) $.
Thus (\ref{string}) implies the string property for $B(\lambda)$.
\medskip

\subsection{}~\label{123}
For completeness we now include Kashiwara's proof of the following 
lemma.
\medskip

\begin{Lemma}~\label{illu1} 
The refined Demazure formula (\ref{demazureformel}) follows from the
string property for $ B(\lambda) $.
\end{Lemma}
\begin{proof}
If $ \tilde{e}_i b = 0 $ for $ b \in B(\la ) $ then clearly $ { \mathcal D }_i b $
is an $ i $-string having $ b $ as its highest weight vector. Moreover, 
an easy 
calculation shows that $ { \mathcal D }_i S = S $ for $ S $ any $i$-string. 
Now Theorem 2 of [K2] says that 
\begin{equation}
B(\la ) = \bigcup_{ k_i \geq 0, j_i \in I, m \geq 0 } 
\tilde{f}^{ k_m}_{j_m} \tilde{f}^{ k_{m-1}}_{j_{m-1}} \ldots 
\tilde{f}^{ k_{1}}_{j_{1}} u_{\la} .
\end{equation} 
Hence, $ B(\la ) $ is the disjoint union of $ i $-strings for any $ i
\in I $, since 
$i$-strings are either disjoint or coincide. 

\medskip
We now prove (\ref{demazureformel}) by induction on $ l(w)$. 
We thus assume the formula for 
$ s_{i_n} w = s_{i_{n-1}}  s_{i_{n-2}} \ldots  s_{i_{1}} $ and need to 
check the equality 
\begin{equation}~\label{equality}
\sum_{b \in B_w (\la )  } b = {\mathcal D }_{i_n} 
\Big( \sum_{b \in B_{s_{i_n} w} (\la )} b \, \Big).
\end{equation}
As $ {\mathcal D }_i $ leaves any $ i $-string invariant it is enough to 
verify the following equality
\begin{equation}~\label{stringequality}
\sum_{b \in B_w (\la ) \cap S } b = {\mathcal D }_{i_n} 
\Big( \sum_{b \in B_{s_{i_n}w }(\la ) \cap S } b \, \Big) .
\end{equation}
Now (\ref{string}) severely restricts the shape of these intersections, and 
even further restrictions are imposed by the condition 
$$ B_w ( \la ) \cap S = \bigcup_{k} \tilde{f}^k_{i_n} ( B_{s_{i_n} w} (\la)
\cap S ), $$
which is a consequence of the definitions. All together, we are left with 
only three possibilities, namely 
\begin{enumerate} 
\item $ B_w (\la ) \cap S = B_{s_{i_n}w} ( \la ) \cap S = \emptyset $,
\item $ B_w (\la ) \cap S = B_{s_{i_n}w} ( \la ) \cap S = S $,
\item $ B_w (\la ) \cap S = S \text{  and  } 
 B_{s_{i_n}w} ( \la ) \cap S = \{b\}  \text{  where } \tilde{e}_i b = 0 .$
\end{enumerate}

\noindent
In all three cases it is straightforward to 
check that (\ref{stringequality}) holds true. 
\end{proof}

We have thus reduced ourselves to the verification of (1), (2) and (3) of 
\ref{threeproperties}. Kashiwara proves (1) and (2) by realizing the $ B_w (\la)$'s as 
crystals of the Demazure modules whereas the 
proof of the string property (3) relies on the combinatorial 
properties of the operators $ \tilde{e}_i^{\ast} $ and $ \tilde{f}_i^{\ast} $
together with (1) and (2).

\medskip 
Here we shall demonstrate that (1) and (2) can be obtained in the same
combinatorial spirit that is employed for (3), that is without relying on 
an interpretation of $ B_w ( \la ) $'s as crystals for any modules. 
Now it is known that Littelmanns's Path model is equivalent to 
the crystal combinatorics, see eg. [J1] and references therein, and 
that (2) and (3) (which suffice to obtain the string property
(\ref{string}))
can be obtained in that setting, [L2]. Still, Joseph remarks on page 
181 in [J1] that it seems extremely difficult to establish (2) 
purely combinatorially.

\section{Properties of $B_w( \infty ) $ }

\subsection{}

Recall the injective morphism $ \Psi_i : B(\infty ) \rightarrow 
B(\infty ) \otimes B_i $ from the previous section.
Using its  properties (\ref{eq1}), (\ref{eq2}) and (\ref{eq3}) one can obtain information 
about the commutation of $ \tilde{f}_i $ and $ \tilde{e}_i $; this is
illustrated by the following lemma.

\medskip

\begin{Lemma}~\label{illu} For any $ i,j \in I $ and $ b \in B(\infty ) $ we have
$$ \bigcup_{k,n} \tilde{f}_i^n \tilde{f}_j^{ \ast k } b = 
   \bigcup_{k,n}  \tilde{f}_j^{ \ast k } \tilde{f}_i^n b. $$
\end{Lemma}
\begin{proof}
If $ i \not= j $ then by Corollary 2.2.2 of [K1] $ \tilde{f}_i $ and 
$ \tilde{f}_j^{\ast} $ commute and there is nothing to prove. So we 
assume $ i=j $. Write 
$$ \Psi_i ( b ) = b_0 \otimes \tilde{f}_i^m b_i, $$ 
and let $ \va := \va_i ( b_0 ) $ and $ \ve := m $. Now, $ \Psi_i $ is an 
embedding so to show the equality of the lemma it is enough to see that 
both sides have the same image under $ \Psi_i $. So we replace $ b $ by 
$ b_0 \otimes \tilde{f}_i^m b_i $ and keep in mind that the action of 
$ \tilde{f}_j^{ \ast k } $ is on the right factor while $\tilde{f}_i $ 
acts as on a tensor product. 

\medskip
Let now $ \Psi_i (b) = b_0 \otimes \tilde{f}_i^m b_i $ be represented as
a point in the crystal graph associated to $ B(\infty ) \otimes B_i $. The 
crystal graph is a representation of the action of $ \tilde{f}_i $ on 
$ B(\infty ) \otimes B_i $, so there is an arrow between two points in 
the graph if $ \tilde{f}_i $ carries the corresponding crystal elements 
to each other. 

\medskip

If $ \va \leq m $ the action of $\tilde{f}_i $ is on the second factor 
and there is a horizontal arrow leaving $  b_0 \otimes \tilde{f}_i^m b_i $ 
and if $ \va > m $ there is a vertical arrow leaving 
$  b_0 \otimes \tilde{f}_i^m b_i $  

\medskip

One typically gets a picture as the following one.

\medskip 

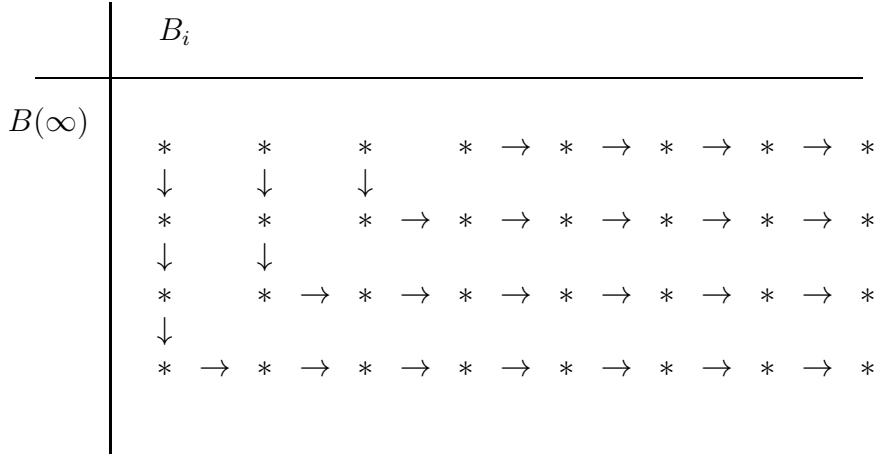
\unitlength0.1cm
\begin{center}
\begin{picture}(10,20)(0,0)
\put (-50,0) {\line(1,0){110}}
\put (-40,10){\line(0,-1){60}}
\put (-35,5) { $ B_i $ }
\put (-55,-7) { $ B(\infty) $ }
\put ( -37,-25) { $ \begin{array}{ccccccccccccccc}
\ast && \ast && \ast && \ast & \rightarrow & \ast & \rightarrow &
\ast & \rightarrow & \ast & \rightarrow & \ast \\
\da && \da && \da && &&&&&&&& \\
 \ast && \ast && \ast & \rightarrow & \ast & \rightarrow &
\ast & \rightarrow & \ast & \rightarrow & \ast &  \rightarrow   & \ast   \\    \da && \da &&  && &&&&&&&& \\ 
\ast && \ast & \rightarrow & \ast & \rightarrow & \ast & \rightarrow &
\ast & \rightarrow & \ast & \rightarrow & \ast &  \rightarrow   & \ast \\
\da &&  &&  && &&&&&&&& \\
\ast & \rightarrow & \ast & \rightarrow & \ast & \rightarrow & 
\ast & \rightarrow &
\ast & \rightarrow & \ast & \rightarrow & \ast &  \rightarrow   & \ast          \end{array} $ }                      
\end{picture}
\end{center}
\vskip5cm

\noindent
The subset of $ B( \la ) $ ,
$$ \bigcup_{k} \tilde{f}_i^k ( b_0 \otimes \tilde{f}_i^m b_i ), $$
is represented by the points of the graph that can be hit by a 
sequence of arrows starting in $ b_0 \otimes \tilde{f}_i^m b_i $. 

\medskip

On the other hand the action of $ \tilde{f}_i^{\ast} $ is always on the 
second factor of the tensor product, so $ \tilde{f}_i^{\ast} $ always 
takes a point in the graph to its right neighbour. Using this information
one can now calculate the two sides of the lemma; in both cases one 
gets the infinite rectangle whose upper left corner is 
$ \Psi_i ( b) = b_0 \otimes \tilde{f}_i^m b_i $ and whose lower left 
corner is the point below $ b_0 \otimes \tilde{f}_i^m b_i $ in which the 
arrows change direction. The lemma is proved.
\end{proof}

\subsection{}

We can use the above to show the following result.
\medskip

\begin{Theorem}~\label{union}
$B_w (\infty ) = \bigcup_{k_1, \ldots k_n} 
\tilde{f}_{i_1}^{\ast k_1} \ldots \tilde{f}_{i_n}^{\ast k_n} u_{\infty}$.
\end{Theorem}

\begin{proof}
By definition $\tilde{f}_{i}^{\ast k} u_{\infty} = 
\tilde{f}_{i}^{ k} u_{\infty} $ for all $ k $ and all $ i $. So we get that
$$B_w (\infty ) = \bigcup_{k_1, \ldots k_n} 
\tilde{f}_{i_n}^{ k_n} \ldots \tilde{f}_{i_2}^{ k_2}
\tilde{f}_{i_1}^{\ast k_1} u_{\infty}.$$
Using Lemma~2.1 we can move $\tilde{f}_{i_1}^{\ast k_1}$ to the front 
position. We then proceed with $\tilde{f}_{i_2}^{ k_2}$ etc. The 
theorem is proved.
\end{proof}

\subsection{}
We can now deduce the property (1) of $ B_w ( \infty ) $:

\begin{Corollary} $ B_w^{\ast}( \infty )  = B_{w^{-1}}(\infty) $.
\end{Corollary}
\begin{proof}
Let $ b \in B_w(\infty) $, i.e. $ b=\tilde{f}_{i_n}^{ k_n} \ldots 
\tilde{f}_{i_1}^{ k_1} u_{\infty}$ for some $ k_1, \ldots k_n $. The 
definition of $ \tilde{f}_i^{\ast} $ then gives that 
$$ b^{\ast} = \tilde{f}_{i_n}^{\ast k_n}
\tilde{f}_{i_{n-1}}^{\ast k_{n-1}}\ldots 
\tilde{f}_{i_1}^{\ast k_1} u_{\infty} $$
But from Theorem \ref{union} we see that $ b^{\ast} \in B_{w^{-1}} (\infty ) $
and the corollary is proved.
\end{proof}

\subsection{}
We shall now consider the property (2). To that end we prove the 
following lemma
\medskip
\begin{Lemma}~\label{ecomm} 
For all $i,j \in I  $ and for all $ b \in B(\infty ) $ 
we have that
$$ \tilde{e}_i  \, \bigcup_k \tilde{f}_j^{\ast k} b \,\subset \, 
\bigcup_k \tilde{f}_j^{\ast k} \tilde{e}_i b \, \cup \, 
\bigcup_k \tilde{f}_j^{\ast k} b \, \cup \, \{ 0 \} $$
\end{Lemma}

\begin{proof}
Again only the case $ i= j $ is nontrivial; otherwise $ \tilde{e}_i $ 
and $ \tilde{f}_j^{\ast } $ commute. We apply the morphism $ \Psi_i $ to 
both sides of the lemma and can then check the inclusion in the crystal 
graph:

\medskip 

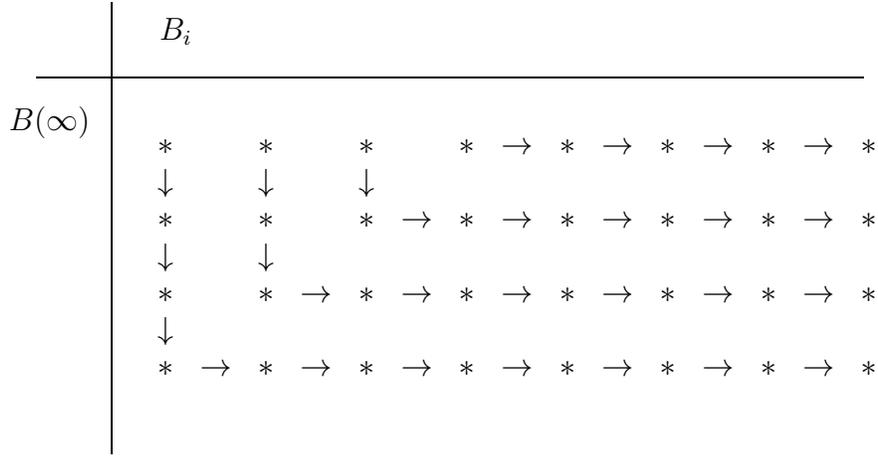
\unitlength0.1cm
\begin{center}
\begin{picture}(10,10)(0,0)
\put (-50,0) {\line(1,0){110}}
\put (-40,10){\line(0,-1){60}}
\put (-35,5) { $ B_i $ }
\put (-55,-7) { $ B(\infty) $ }
\put ( -37,-25) { $ \begin{array}{ccccccccccccccc}
\ast && \ast && \ast && \ast & \rightarrow & \ast & \rightarrow &
\ast & \rightarrow & \ast & \rightarrow & \ast \\
\da && \da && \da && &&&&&&&& \\
 \ast && \ast && \ast & \rightarrow & \ast & \rightarrow &
\ast & \rightarrow & \ast & \rightarrow & \ast &  \rightarrow   & \ast   \\    \da && \da &&  && &&&&&&&& \\ 
\ast && \ast & \rightarrow & \ast & \rightarrow & \ast & \rightarrow &
\ast & \rightarrow & \ast & \rightarrow & \ast &  \rightarrow   & \ast \\
\da &&  &&  && &&&&&&&& \\
\ast & \rightarrow & \ast & \rightarrow & \ast & \rightarrow & 
\ast & \rightarrow &
\ast & \rightarrow & \ast & \rightarrow & \ast &  \rightarrow   & \ast          \end{array} $ }                      
\end{picture}
\end{center}
\vskip5cm

The graph is infinite to the right. We understand that 
$ \tilde{e}_i b = 0 $ if there is no arrow ending at the point corresponding
to $ b $. Again, $ \tilde{f}^{\ast}_i $ acts by shifting a $ b $ to the 
right while $ \tilde{f}_i $ follows the arrows (and hence $ \tilde{e_i} $ 
follows the arrows in negative direction).

\medskip

Let us start out by verifying that there are no points missing in the 
above picture. So we must check that if the arrow leaving $ b $ is 
vertical and there is no arrow ending at $b$ then neither should there 
be any arrow ending at $ b$'s right neighbour.

\medskip

Let thus $ b $ be as indicated and write $ \Psi_i ( b ) = b_0 \otimes 
\tilde{f}_i^m b_i $. Then $ \va ( b_0 ) > \ve ( \tilde{f}^m_i b_i ) = m $
because the arrow leaving $ b $ is vertical. Now $ \tilde{e}_i ( b)  = 0 $ 
implies that $ \tilde{e}_i (b_0 ) = 0 $ because $ \Psi_i $ commutes with 
$ \tilde{e}_i $ and no element of $ B_i $ is mapped to $ 0 $ under 
$ \tilde{e}_i $. Since $ \va ( b_0 ) \geq 
\ve ( \tilde{f}_i^{m+1} b_i ) $ we indeed get that 
$$ \tilde{e}_i ( \Psi_i ( \tilde{f}_i^{\ast} b )) = 
\tilde{e}_i ( b_0 \otimes \tilde{f}_i^{m+1} b_i ) = 
\tilde{e}_i  b_0 \otimes \tilde{f}_i^{m+1} b_i  = 0 $$

\medskip

We now split the verification of the lemma into several cases. 
Firstly we 
consider the case of a $ b $ with $ \tilde{e}_i (b) = 0 $.
Then the left hand side of the lemma consists of those points in the row of 
$ b $ from which a horizontal arrow is leaving. But this is contained in 
the right hand side of the lemma.


\medskip

Then we consider the case of a vertical arrow entering and 
a vertical arrow leaving $ b $. In that case the left hand side of the lemma
consists of all the 
points that are positioned to the right of $ b$ (including $ b $ itself)
together with the points in the row above $ b $ that have an arrow leading 
into one of the first points. In addition, the right hand side consists of the 
first points together with their upper neighbours. Thus the inclusion
also holds in this case.

\medskip

We then consider the case of a vertical arrow entering and a
horizontal arrow leaving $b$.
Then the left hand side of the lemma consists of the points positioned to the
right of $b$ together 
with $b$ itself and its immediate predecessor. This is contained in the right 
hand side of the lemma (only the $ k = 0 $ part of the first union is needed).

\medskip

Finally we consider the case of horizontal arrows entering as well 
as leaving $b$. In that case the left hand side consists of all points
to the right of $b$ together with $b$'s immediate predecessor, which
is included in the right hand side (only the first union is needed).
\end{proof}

\subsection{}

We can now show the property (2) of $ B_w(\infty ) $:

\begin{Theorem} For $ i \in I $ we have $ \tilde{e}_i B_w ( \infty ) 
\subset  B_w ( \infty ) \cup\{ 0 \} $ 
\end{Theorem}

\begin{proof}
We argue by induction on $ l(w) $ and thus assume the theorem for 
$ l( w ) -1 $. By Theorem \ref{union} 
$ B_w (\infty ) $ satisfies the equality 
$$ B_w ( \infty ) = \bigcup_{ k_1 } \tilde{f}_{i_1}^{\ast k_i } 
B_{ ws_{i_1} } ( \infty ) $$
By induction hypothesis $ \tilde{e}_i B_{ ws_{i_1} } ( \infty ) \subset 
B_{ ws_{i_1} } ( \infty ) \cup\{ 0 \} $. Combining this with lemma
\ref{ecomm} we obtain the induction step. The theorem is proved. 
\end{proof}

\section{The Braid Relations}
In this section we verify that the crystal Demazure operators ${\mathcal D }_i$
satisfy the braid relations on dominant weights. From this
it follows that $ B_w(\la) $ is independent of the choice 
of reduced expression for $w$. 
Note that Kashiwara 
has observed that the ${\mathcal D }_i$ do {\it not} satisfy 
the braid relations in general. 

\subsection{}
Since $W$ is a Weyl group, it is enough to check the
braid relations for $W$ of type $\mbox{A}_2, \mbox{B}_2 $ or $ \mbox{G}_2 $.
Indeed, for $$ w= w_1 s_{i_k} s_{i_{k-1}}  s_{i_k} w_2 = 
w_1  s_{i_{k-1}} s_{i_k}s_{i_{k-1}} w_2 $$ 
a braid relation of type $\mbox{A}_2 $ it is enough to
check the case $ w_1= 1$.
By the refined sum formula (\ref{demazureformel}) applied to $ w_2 $
one should then show that 
\begin{equation}{~\label{braid}}
{\mathcal D}_{i_k} {\mathcal D}_{i_{k-1}} 
{\mathcal D}_{i_k}\left( \sum_{b \in B_{w_2}(\la)} b\right) = 
{\mathcal D}_{i_{k-1}} {\mathcal D}_{i_{k}} {\mathcal D}_{i_{k-1}} \left(
\sum_{b \in B_{w_2}(\la)} b \right) .
\end{equation}
Using (\ref{demazureformel}) once more, the left hand side of this is the 
sum over all elements of $ B_{ s_k s_{k-1} s_{k} w_2 }(\la) $ while the right hand 
side is the sum over the elements of $ B_{ s_{k-1} s_{k} s_{k-1}w_2 }(\la) $.
We write $ w_2 = s_{i_l} s_{i_{l-1}} \ldots s_{i_1} $ and get then by repeated use of
Lemma~\ref{illu} like in Theorem~\ref{union} that
$$ B_{ s_k s_{k-1} s_k w_2 }(\la) = 
\bigcup_{k_1, \ldots k_l } 
\tilde{f}_{i_1}^{\ast k_1} \ldots \tilde{f}_{i_{l-1}}^{ \ast k_{l-1}}
\tilde{f}_{i_l}^{ \ast k_l} B_{ s_k s_{k-1} s_k  }(\la). $$
Similarly, the right hand side of (\ref{braid}) is the sum over 
$$\bigcup_{k_1, \ldots k_l } 
\tilde{f}_{i_1}^{\ast k_1} \ldots \tilde{f}_{i_{l-1}}^{ \ast k_{l-1}}
\tilde{f}_{i_l}^{ \ast k_l} B_{ s_k s_{k-1} s_k  }(\la). $$
The $ \mbox{A}_2 $-case of the braid relations then implies 
(\ref{braid}).
Similarly, one reduces the other braid relations to rank 2 cases.




\subsection{}
To check the $\mbox{A}_2, \mbox{B}_2 $ or $ \mbox{G}_2 $ cases, 
we appeal to the representation theoretical 
interpretation of $ B(\la) $ as basis at $ q=0 $ of the 
irreducible highest weight module $ V(\la) $ for the quantum
group $ U_q({\mathfrak g} ) $. 

Let us consider the $\mbox{A}_2 $-case
and write $ \la = ( \la_1, \la_2) $
in terms of the fundamental weights $ (\omega_1, \omega_2 ) $. 
Then 
$ \tilde{f}_1^{\la_2} u_{\la} $ is nonzero, since it is the lowest 
element of the $2$-string with highest element $ u_{\la} $. But
by weight considerations 
$ \tilde{f}_1^{\la_2} u_{\la} $ must be mapped to $ 0 $ under $ \tilde{e}_1 $ 
and therefore it is the highest element of the $ 1 $-string, whose
lowest element is 
$ \tilde{f}_2^{\la_1+\la_2} \tilde{f}_1^{\la_1} u_{\la} $ and 
especially nonzero. Continuing, we find that 
$$ \tilde{f}_1^{\la_2} \tilde{f}_2^{\la_1+\la_2} \tilde{f}_1^{\la_1} u_{\la}
\in B_{s_{1} s_{2} s_{1}}(\la)  \subset B(\la)  $$ is nonzero.
The lowest weight vector space of $ V(\la) $ is one
dimensional and so this element is the unique lowest element if $B(\la) $ .

Now, by (2) of (\ref{threeproperties}),
$ B_{s_{1} s_{2} s_{1}}(\la) $ is invariant under all the 
$ \tilde{e}_i $ operators. Since it moreover contains the 
lowest element, it must be equal to all of $ B(\la) $.
The same conclusion holds for $ B_{s_{2} s_{1} s_{2}}(\la) $ and 
then $ B_{s_{2} s_{1} s_{2}}(\la) =  B_{s_{1} s_{2} s_{1}}(\la) $ as 
promised.

\end{document}